\documentclass{amsart}

\RequirePackage{fix-cm}
\usepackage{graphicx}
\usepackage{amsmath, amsopn,amstext,amscd,amsfonts,amssymb}
\usepackage{dsfont}
\usepackage{comment}
\usepackage[active]{srcltx}
\usepackage{graphicx, epsfig, subfig}

\usepackage{textcomp}

\usepackage{algorithm}

\makeatletter
\def\BState{\State\hskip-\ALG@thistlm}
\makeatother

\def\downbar#1{
\setbox10=\hbox{$#1$}
            \dimen10=\ht10 \advance\dimen10 by 2.5pt
            \ifdim \dimen10<15pt 
               \advance\dimen10 by -0.5pt
               \dimen11=\dimen10
               \advance\dimen10 by 2.5pt
               \lower \dimen11
            \else \lower \ht10 \fi
            \hbox {\hskip 1.5pt \vrule height \dimen10 depth \dp10}}
\def\upbar#1{
\setbox10=\hbox{$#1$}
            \dimen10=\ht10 \advance\dimen10 by \dp10 \advance\dimen10 by 2.5pt
            \ifdim \dimen10<15pt 
                \advance\dimen10 by 2pt \fi
            \raise 2.5pt \hbox {\hskip -1.5pt \vrule height \dimen10}}

\usepackage{multicol}
\usepackage{colortbl}
\usepackage{exscale}
\usepackage{amssymb,latexsym,amsthm,amsfonts,color,fancyhdr,mathrsfs}
\usepackage{lscape}
\usepackage{rotating}
\usepackage{caption}

\newtheorem{theorem}{\bf Theorem}[section]

\newtheorem{lemma}{\bf Lemma}[section]

\newtheorem{corollary}{\bf Corollary}[section]
\newtheorem{remark}{\bf Remark}[section]

\numberwithin{equation}{section}
\begin{document}
\title[Orthogonal polynomials]{Characterization of Orthogonal Polynomials on lattices}

\author{D. Mbouna}
\address{University of Almer\'ia, Department of Mathematics, Almer\'ia, Spain}
\email{mbouna@ual.es}
\author{Juan F. Ma\~nas-Ma\~nas}
\address{University of Almer\'ia, Department of Mathematics, Almer\'ia, Spain}
\email{jmm939@ual.es}
\author{Juan J. Moreno-Balc\'azar}
\address{University of Almer\'ia, Department of Mathematics, Almer\'ia, Spain}
\email{balcazar@ual.es}


\subjclass[2010]{42C05, 33C45}
\date{\today}
\keywords{semiclassical functional, Wilson polynomials, continuous dual Hahn polynomials, lattices}

\begin{abstract}
We consider two sequences of orthogonal polynomials $(P_n)_{n\geq 0}$ and $(Q_n)_{n\geq 0}$ such that 
\begin{align*}
\sum_{j=1} ^{M} a_{j,n}\mathrm{S}_x\mathrm{D}_x ^k P_{k+n-j} (z)=\sum_{j=1} ^{N} b_{j,n}\mathrm{D}_x ^{m} Q_{m+n-j} (z)\;,
\end{align*}
with $k,m,M,N \in \mathbb{N}$, $a_{j,n}$ and $b_{j,n}$ are sequences of complex numbers, $$2\mathrm{S}_xf(x(s))=(\triangle +2\,\mathrm{I})f(z),~~ \mathrm{D}_xf(x(s))=\frac{\triangle}{\triangle x(s-1/2)}f(z),$$ $z=x(s-1/2)$, $\mathrm{I}$ is the identity operator, $x$ defines a lattice, and $\triangle f(s)=f(s+1)-f(s)$. We show that under some natural conditions, both involved orthogonal polynomials sequences $(P_n)_{n\geq 0}$ and $(Q_n)_{n\geq 0}$ are semiclassical whenever $k=m$. Some particular cases are studied closely where we characterize the continuous dual Hahn and Wilson polynomials for quadratic lattices.
\end{abstract}
\maketitle

\section{Introduction}\label{introduction}
One of motivations of this work is to obtain a Al-Salam and Chihara type (see \cite{Al-Salam-1972}) characterization of classical orthogonal polynomials on lattices. That is to characterize all orthogonal polynomials sequences (OPS), $(P_n)_{n\geq 0}$, solutions of the following equation
\begin{align}\label{problem}
&(az^2+bz+c)\frac{\triangle}{\triangle x(s-1/2)} P_n(x(s-1/2))\\
\nonumber&\qquad=(\triangle +2\,\mathrm{I})(a_nP_{n+1}+b_nP_n+c_nP_{n-1})(x(s-1/2)),
\end{align}
where $\,\mathrm{I}$ is the identity operator, $a$, $b$ and $c$ are some well chosen complex numbers, $x$ defines a class of lattices (or grids) with, generally, nonuniform step-size, $\triangle f(s)=f(s+1)-f(s)$, and $\nabla f(s)=\triangle f(s-1)$. This problem finds his origin in \cite{KCDMJP2021-a}. The case where the lattice $x$ is $q$-quadratic and given by $x(s)=(q^{-s}+q^s)/2$ was solved recently in \cite{DMAZ2022}, where under some conditions imposed in $a$, $b$ and $c$,  the only solutions are the Askey-Wilson polynomials including special or limiting cases of them. But as noticed in \cite{KJ2018}, when we consider a quadratic lattice for \eqref{problem}, solutions ``can not easily be deduced from those of Askey-Wilson polynomials". Our objective is to present an analogue of this problem for quadratic lattices $x(s)=\mathfrak{c}_4s^2+\mathfrak{c}_5s$ and therefore to provide another characterization of such OPS on lattices. For instance, it is well known that classical OPS on lattices are characterized (see \cite[Theorem 4.3]{CSFM2009}) by the following equation
$$(\triangle +2\,\mathrm{I}) P_n(x(s-1/2))=\frac{\triangle}{\triangle x(s-1/2)}(a_nP_{n+1}+b_n\mathrm{D}_xP_n+c_n\mathrm{D}_xP_{n-1})(x(s-1/2))\;.$$
In addition it is proved in \cite{DMAZ2022} that, for the $q$-quadratic lattice $x(s)=(q^{-s}+q^s)/2$, the following equation has classical OPS as solutions
\begin{align*}
\sum_{j=n-1} ^{n+1} a_{n,j}(\triangle +2\,\mathrm{I})P_j(x(s-1/2))=\sum_{j=n-4} ^{n+2} b_{n,j}\frac{\triangle}{\triangle x(s-1/2)} P_{j} (x(s-1/2)) \;.
\end{align*} 
Therefore a second motivation of this work is to study such structure relations in a more general form. This is why we consider two (monic) OPS $(P_n)_{n\geq 0}$ and $(Q_n)_{n\geq 0}$, and $k,m,M,N \in \mathbb{N}$ such that 
\begin{align*}
\sum_{j=1} ^{M} a_{j,n}\mathrm{S}_x\mathrm{D}_x ^k P_{k+n-j} (z)=\sum_{j=1} ^{N} b_{j,n}\mathrm{D}_x ^{m} Q_{m+n-j} (z), \quad z=x(s)\;,
\end{align*}
where $2\mathrm{S}_xf(x(s))=(\triangle +2\,\mathrm{I})f(x(s-1/2))$, $\mathrm{D}_xf(x(s))=\frac{\triangle}{\triangle x(s-1/2)}f(x(s-1/2))$, $a_{j,n}$ and $b_{j,n}$ are sequences of complex numbers.  Our aim is to study the semiclassical character of the OPS involved in the above equation.

The structure of this note is as follows. Section 2 presents some basic facts of the algebraic theory of OPS on lattices and some preliminary results. In Section \ref{main-results} our main results are stated and proved. In Section \ref{Special_case} we present a finer result for a special case.

\section{Background and Preliminary}
Let $\mathcal{P}$ be the vector space of all polynomials with complex coefficients
and let $\mathcal{P}^*$ be its algebraic dual. A simple set in $\mathcal{P}$ is a sequence $(Q_n)_{n\geq0}$ such that $\mathrm{deg}(Q_n)=n$ for each $n$. A simple set $(Q_n)_{n\geq0}$ is called an OPS with respect to ${\bf w}\in\mathcal{P}^*$ if 
$$
\langle{\bf w},Q_nQ_m\rangle=h_n\delta_{n,m}\;,\quad m=0,1,\ldots;\;h_n\in\mathbb{C}\setminus\{0\}\;.
$$
In this case, we say that ${\bf w}$ is  regular. The left multiplication of a functional ${\bf w}$ by a polynomial $\pi$ is defined by
$$
\left\langle \pi {\bf w}, p  \right\rangle =\left\langle {\bf w},\pi p  \right\rangle\;, \quad p\in \mathcal{P}\;.
$$
A dual basis $({\bf r}_n)_{n\geq 0}$ of a simple set polynomial sequence $(Q_n)_{n\geq 0}$ is a sequence in $\mathcal{P}^*$ such that $\left\langle {\bf r}_n, Q_m   \right\rangle =\delta_{n,m}$, for all $n,m$. Consequently, if $(Q_n)_{n\geq0}$ is a (monic) OPS with respect to ${\bf w}\in\mathcal{P}^*$, then the corresponding dual basis is explicitly given by 
\begin{align}\label{expression-an}
{\bf r}_n =\left\langle {\bf w} , Q_n ^2 \right\rangle ^{-1} Q_n{\bf w}.
\end{align}
In addition any functional ${\bf v} \in \mathcal{P}^*$ (when $\mathcal{P}$ is endowed with an appropriate strict inductive limit topology, see \cite{M1991}) can be written in the sense of the weak topology in $\mathcal{P}^*$ as 
\begin{align*}
{\bf v} = \sum_{n=0} ^{\infty} \left\langle {\bf v}, Q_n \right\rangle {\bf r}_n.
\end{align*}
It is known (see \cite{C1978}) that a monic OPS, $(Q_n)_{n\geq 0}$, is characterized by the following three-term recurrence relation (TTRR):
\begin{align}\label{TTRR_relation}
Q_{-1} (z)=0, \quad Q_{n+1} (z) =(z-B_n)Q_n (z)-C_n Q_{n-1} (z)\;, \quad C_n \neq 0\;,
\end{align}
and, therefore,
\begin{align}\label{TTRR_coefficients}
B_n = \frac{\left\langle {\bf w} , zQ_n ^2 \right\rangle}{\left\langle {\bf w} , Q_n ^2 \right\rangle},\quad C_{n+1}  = \frac{\left\langle {\bf w} , Q_{n+1} ^2 \right\rangle}{\left\langle {\bf w} , Q_n ^2 \right\rangle}.
\end{align}

In our framework, a lattice $x$ is a mapping given by (see \cite{ARS1995})
\begin{align*}
x(s):=\left\{
\begin{array}{lcl}
\mathfrak{c}_1 q^{-s} +\mathfrak{c}_2 q^s +\mathfrak{c}_3,&  q\neq1\\ [7pt]
\mathfrak{c}_4 s^2 + \mathfrak{c}_5 s +\mathfrak{c}_6, &  q =1,
\end{array}
\right.
\end{align*}
where $q>0$ and $\mathfrak{c}_j$ ($1\leq j\leq6$) are complex numbers such that $(\mathfrak{c}_1,\mathfrak{c}_2)\neq(0,0)$ if $q\neq1$.
Note that
$x\big(s+\frac12\big)+x\big(s-\frac12\big)=2\alpha x(s)+2\beta,$
where
\begin{align*}
\alpha=\frac{q^{1/2}+q^{-1/2}}{2},\quad
\beta=\left\{
\begin{array}{lcl}
(1-\alpha)\mathfrak{c}_3, &  q\neq1,\\ [7pt]
\mathfrak{c}_4/4, &  q =1.
\end{array}
\right.
\end{align*}
We define $\alpha_n:=(q^{n/2} +q^{-n/2})/2$ and
\begin{align*}
\gamma_n := \left\{
\begin{array}{lcl}
\displaystyle\frac{q^{n/2}-q^{-n/2}}{q^{1/2}-q^{-1/2}}, & q\neq1 \\ [7pt]
n, & q=1.
\end{array}
\right. 
\end{align*}
We set $\gamma_{-1}:=-1$ and $\alpha_{-1}:=\alpha$. We also define two operators $\mathrm{D}_x$ and $\mathrm{S}_x$ on $\mathcal{P}$ by 
\begin{align*}
\mathrm{D}_x f(x(s))=\frac{\triangle}{{ \triangle} x(s-1/2)}f(x(s-1/2)),\quad \mathrm{S}_x f(x(s))= \frac{1}{2}(\triangle+2\,\mathrm{I})f(x(s-1/2)),
\end{align*}
These operators  induce two elements on $\mathcal{P}^*$, say $\mathbf{D}_x$ and $\mathbf{S}_x$, via the following definition (see \cite{FK-NM2011}): 
\begin{align*}
\langle \mathbf{D}_x{\bf u},f\rangle=-\langle {\bf u},\mathrm{D}_x f\rangle,\quad \langle\mathbf{S}_x{\bf u},f\rangle=\langle {\bf u},\mathrm{S}_x f\rangle.
\end{align*}

 Let $f,g\in\mathcal{P}$ and ${\bf u}\in\mathcal{P}^*$. Then the following properties hold (see e.g. \cite{KCDMJP2021-a, FK-NM2011, SMMFPN2017}):
\begin{align}
\mathrm{D}_x \big(fg\big)&= \big(\mathrm{D}_x f\big)\big(\mathrm{S}_x g\big)+\big(\mathrm{S}_x f\big)\big(\mathrm{D}_x g\big), \label{def-Dx-fg} \\[7pt]
\mathrm{S}_x \big( fg\big)&=\big(\mathrm{D}_x f\big) \big(\mathrm{D}_x g\big)\texttt{U}_2  +\big(\mathrm{S}_x f\big) \big(\mathrm{S}_x g\big), \label{def-Sx-fg} \\[7pt]
f\mathrm{D}_xg&=\mathrm{D}_x\left[ \Big(\mathrm{S}_xf-\frac{\texttt{U}_1}{\alpha}\mathrm{D}_xf \Big)g\right]-\alpha ^{-1}\mathrm{S}_x \Big(g\mathrm{D}_x f\Big) , \label{def-fDxg} \\[7pt]
{\bf D}_x (f{\bf u})&=\Big(\mathrm{S}_xf-\alpha ^{-1} \texttt{U}_1\mathrm{D}_xf \Big){\bf D}_x {\bf u}+\alpha ^{-1}\mathrm{D}_xf~{\bf S}_x{\bf u},\label{def_D_xfu}\\[7pt]
{\bf S}_x (f{\bf u})&=\Big(\alpha \texttt{U}_2 -\alpha^{-1}\texttt{U}_1 ^2 \Big)\mathrm{D}_x f~{\bf D}_x{\bf u} +\Big(\mathrm{S}_xf+\alpha ^{-1} \texttt{U}_1\mathrm{D}_xf \Big){\bf S}_x{\bf u},\label{def_S_xfu}\\[7pt]
f{\bf D}_x {\bf u}&={\bf D}_x\left(S_xf~{\bf u}  \right)-{\bf S}_x\left(D_xf~{\bf u}  \right), \label{def-fD_x-u}\\[7pt]
\alpha \mathbf{D}_x ^n \mathbf{S}_x {\bf u}&= \alpha_{n+1} \mathbf{S}_x \mathbf{D}_x^n {\bf u}
+\gamma_n \texttt{U}_1\mathbf{D}_x^{n+1}{\bf u}, \label{DxnSx-u} 
\end{align}
where
\begin{align*}
\texttt{U}_1 (z)&=\left\{
\begin{array}{lcl}
(\alpha^2-1)\big(z-\mathfrak{c}_3\big), & q\neq1\\[7pt]
2\beta, &q=1,
\end{array}
\right. \\[7pt]
\texttt{U}_2(z)&=\left\{
\begin{array}{lcl}
(\alpha^2-1)\big((z-\mathfrak{c}_3)^2-4\mathfrak{c}_1\mathfrak{c}_2\big), & q\neq1\\[7pt]
4\beta (z-\mathfrak{c}_6)+\mathfrak{c}_5^2/4, & q=1.
\end{array}
\right.
\end{align*}
It is known that (see \cite{KCDMJP2021-a}) if $x(s)=4\beta s^2 +\mathfrak{c}_5s$, then
\begin{align}\label{Dx-xnSx-xn-quadratic}
\mathrm{D}_x z^n =n z^{n-1}+v_nz^{n-2}+\cdots,\quad 
\mathrm{S}_x z^n =z^n+\widehat{v}_nz^{n-1}+\cdots\;,
\end{align}
with $v_n=\beta n(n-1)(2n-1)/3$ and $\widehat{v}_n=\beta n(2n-1)$. 

We denote by $Q_n ^{[k]}$, with $k=0,1,\ldots$, the monic polynomial of degree $n$ defined by
\begin{align*}
Q_n ^{[k]} (z)=\frac{\gamma_{n} !}{\gamma_{n+k} !} \mathrm{D}_x ^k Q_{n+k} (z)\;,
\end{align*}
with $\gamma_0 !=1$, $\gamma_{n+1}!=\gamma_1\cdots \gamma_n \gamma_{n+1}$. If $({\bf r}^{[k]} _n)_{n\geq 0}$ is the dual basis associated to the sequence $(Q_n ^{[k]})_{n\geq 0}$, it is known that (see \cite{KCDMJP2021-a})
\begin{align}
{\bf D}_x ^k {\bf r}^{[k]} _n=(-1)^k \frac{\gamma_{n+k}!}{\gamma_n ! }{\bf r}_{n+k}\;,\quad k=0,1,\ldots\;. \label{basis-Dx-derivatives}
\end{align}
\section{main results}\label{main-results}
We start with the following lemma using ideas developed in \cite{Petronilho2006} and \cite{CD20}.
\begin{lemma}\label{lemma1}
Let $({\bf u}, {\bf v})$ be a pair of regular functionals and $((P_n)_{n\geq 0}, (Q_n)_{n\geq 0})$ the corresponding pair of monic OPS. Assume that for some $k,m,M,N \in \mathbb{N}$, we have 
\begin{align}\label{general_problem}
\sum_{j=0} ^{M} a_{j,n}\mathrm{S}_x P_{n-j} ^{[k]} (z)=\sum_{j=0} ^{N} b_{j,n} Q_{n-j} ^{[m]} (z)\;, \quad \quad n=0,1,\ldots\;,
\end{align}
for some complex sequences $a_{i,n}$ and $b_{i,n}$, with $a_{0,n}=1=b_{0,n}$ and $a_{M,n}b_{N,n}\neq 0$ for all $n$. Let $\mathcal{A}_{M+N}=\Big[l_{i,j}\Big]_{i,j=0} ^{M+N-1}$ be the following matrix of order $M+N$,
\begin{align*}
l_{i,j} = \left\{
\begin{array}{lclcl}
\displaystyle a_{j-i,j}, ~ \textit{if}~~0\leq i\leq N-1~\textit{and}~ i\leq j\leq M+i\;,  \\ [7pt]
b_{j-i+N,j}, ~ \textit{if}~~N\leq i\leq M+N-1~\textit{and}~ i-N\leq j\leq i \;,\\
0,~ \textit{otherwise\;.}
\end{array}
\right. 
\end{align*}
Assume that $det(\mathcal{A}_{M+N})\neq 0$ and $k\geq m$. Then there exist three polynomials $\psi_{N+k+n}$, $\phi_{M+m+n+1}$ and $\rho_{M+m+n}$ of degrees $N+k+n$, $M+m+n+1$ and $M+m+n$, respectively, such that 
\begin{align}\label{relation_with_dual_basis}
\psi_{N+k+n}{\bf u}= {\bf D}_x ^{k-m}\Big(\phi_{M+m+n+1}{\bf D}_x{\bf v} +\rho_{M+m+n}{\bf S}_x{\bf v}  \Big)\;, ~\quad\quad~n=0,1,\ldots\;.
\end{align} 
\end{lemma}

\begin{proof}
Define 
$$R_n(z)= \sum_{j=0} ^{M} a_{j,n} P_{n-j} ^{[k]} (z)\;, ~~\quad n=0,1,\ldots\;.$$
Let $({\bf a}_n)_{n\geq 0}$, $({\bf b}_n)_{n\geq 0}$, $({\bf a}_n ^{[k]})_{n\geq 0}$, $({\bf b}_n ^{[m]})_{n\geq 0}$ and $({\bf r}_n)_{n\geq 0}$ be the associated dual basis to the sequences $(P_n)_{n\geq 0}$, $(Q_n)_{n\geq 0}$, $(P_n ^{[k]})_{n\geq 0}$, $(Q_n ^{[m]})_{n\geq 0}$ and $(R_n)_{n\geq 0}$, respectively.
We are going to prove that 
\begin{align}
{\bf a}_n ^{[k]} = \sum_{l=n} ^{n+M} a_{l-n,l}{\bf r}_l \;,\quad {\bf S}_x{\bf b}_n ^{[m]} = \sum_{l=n} ^{n+N} b_{l-n,l}{\bf r}_l\;,  \quad n=0,1,\ldots\;.\label{expressions_ank_Sxbnm}
\end{align}
Indeed, by definition of $R_n$, we obtain
\begin{align*}
\left\langle {\bf a}_n ^{[k]},R_l  \right\rangle=\sum_{i=0} ^M a_{i,l}\left\langle {\bf a}_n ^{[k]},P_{l-i} ^{[k]}  \right\rangle =\left\{
    \begin{array}{ll}
        a_{l-n,l}\;, & \mbox{if } n\leq l\leq n+M \;,\\
        0\;, & \mbox{otherwise\;.}
    \end{array}
\right.
\end{align*}
Similarly, using \eqref{general_problem}, we write
\begin{align*}
\left\langle {\bf S}_x {\bf b}_n ^{[m]},R_l  \right\rangle=\sum_{i=0} ^N b_{i,l}\left\langle {\bf b}_n ^{[m]},Q_{l-i} ^{[m]}  \right\rangle =\left\{
    \begin{array}{ll}
        b_{l-n,l}\;, & \mbox{if } n\leq l\leq n+N \;,\\
        0 \;,& \mbox{otherwise\;.}
    \end{array}
\right.
\end{align*}
Therefore \eqref{expressions_ank_Sxbnm} hold by writing 
$${\bf a}_n ^{[k]}=\sum_{l=0} ^{\infty} \left\langle  {\bf a}_n ^{[k]},R_l  \right\rangle {\bf r}_l \;,  \quad {\bf S}_x{\bf b}_n ^{[m]}=\sum_{l=0} ^{\infty} \left\langle  {\bf S}_x{\bf b}_n ^{[m]},R_l  \right\rangle {\bf r}_l\;, $$
and by using what is preceding. Taking $n=0,1,\ldots,N-1$ and $n=0,1,\ldots, M-1$ in the first and in the second equation of \eqref{expressions_ank_Sxbnm}, respectively, we obtain a system of equations whose matrix is $\mathcal{A}_{M+N}$ and since $det(\mathcal{A}_{M+N})\neq 0$, then we may write 
$$r_n= \sum_{i=0} ^{N-1} \widehat{a}_{n,i}{\bf a}_i ^{[k]} +\sum_{j=0} ^{M-1}\widehat{b}_{n,j}{\bf S}_x{\bf b}_j ^{[m]}\;,\quad
n=0,1,\ldots, M+N-1,$$
for some complex sequences $\widehat{a}_{n,i}$ and $\widehat{b}_{n,j}$. Even more we may also write
\begin{align}
\sum_{i=0} ^{n+N} a'_{n,i}{\bf a}_i ^{[k]} =\sum_{j=0} ^{n+M}b'_{n,j}{\bf S}_x{\bf b}_j ^{[m]}\;,\quad n=0,1,\ldots\;.\label{equation_with_basis00}
\end{align}
with $a'_{n,i}$ and $b'_{n,j}$ complex sequences with $a'_{n,n+N}=b_{N,M+N+n}$ and $b'_{n,n+M}=a_{M,M+N+n}$. In addition using successively  \eqref{DxnSx-u}, \eqref{basis-Dx-derivatives}, \eqref{def_D_xfu} and \eqref{def_S_xfu}, we obtain
\begin{align*}
&{\bf D}_x ^m {\bf S}_x {\bf b}_j ^{[m]}\\
&=\frac{\alpha_{m+1}}{\alpha} {\bf S}_x {\bf D}_x ^m {\bf b}_j ^{[m]} +\frac{\gamma_m}{\alpha}\texttt{U}_1{\bf D}_x ^{m+1}{\bf b}_j ^{[m]}\\
&= \frac{(-1)^m\gamma_{m+j}!}{\gamma_j!~\alpha \left\langle {\bf v}, Q_{m+j} ^2  \right\rangle }\Big( \alpha_{m+1}{\bf S}_x (Q_{m+j}{\bf v}) +\gamma_m \texttt{U}_1 {\bf D}_x (Q_{m+j} {\bf v}) \Big) \\
&=\frac{(-1)^m\gamma_{m+j}!}{\gamma_j!~\alpha \left\langle {\bf v}, Q_{m+j} ^2  \right\rangle }\Big[\Big(\Big(\alpha \alpha_{m+1}\texttt{U}_2 -(\alpha_{m+1}+\gamma_m)\frac{\texttt{U}_1 ^2}{\alpha}\Big)\mathrm{D}_xQ_{m+j} +\gamma_m\texttt{U}_1 \mathrm{S}_xQ_{m+j} \Big){\bf D}_x{\bf u} \\
&\quad \quad \quad \quad \quad \quad \quad+\Big(\alpha_{m+1}\mathrm{S}_xQ_{m+j} +(\alpha_{m+1}+\gamma_m)\frac{\texttt{U}_1}{\alpha}\mathrm{D}_xQ_{m+j} \Big){\bf S}_x{\bf u}\Big]\;.
\end{align*}
Now since $k\geq m$, we apply ${\bf D}_x ^k$ to \eqref{equation_with_basis00} using \eqref{basis-Dx-derivatives} to obtain
$$\sum_{i=0} ^{n+N}(-1)^k \frac{\gamma_{k+i}!}{\gamma_i !} a'_{n,i}{\bf a}_{k+i} ={\bf D}_x ^{k-m}\Big(\sum_{j=0} ^{n+M}b'_{n,j}{\bf D}_x ^m{\bf S}_x{\bf b}_j ^{[m]}\Big)\;,\quad n=0,1,\ldots\;. $$
Hence \eqref{relation_with_dual_basis} holds where
\begin{align*}
\psi_{N+k+n}(z)&=\sum_{i=0} ^{n+N} \frac{(-1)^k\gamma_{k+i}!}{\gamma_i !\left\langle {\bf u},P_{k+i} ^2 \right\rangle}a'_{n,i}P_{k+i}(z)\;, \\
\phi_{M+m+n+1}(z)&=\sum_{j=0} ^{n+M} \frac{(-1)^m\gamma_{m+j}!}{\gamma_j!~\alpha \left\langle {\bf v},Q_{m+j} ^2 \right\rangle}b'_{n,j}\Big(  \Big(\alpha \alpha_{m+1}\texttt{U}_2 -(\alpha_{m+1}+\gamma_m)\frac{\texttt{U}_1 ^2}{\alpha}\Big)\mathrm{D}_xQ_{m+j}(z) \\
&\quad \quad \quad \quad \quad\quad \quad \quad \quad \quad \quad+\gamma_m\texttt{U}_1 \mathrm{S}_xQ_{m+j} (z)\Big) \;,\\
\rho_{M+m+n}(z)&= \sum_{j=0} ^{n+M} \frac{(-1)^m\gamma_{m+j}!}{\gamma_j!~\alpha\left\langle {\bf v},Q_{m+j} ^2 \right\rangle}b'_{n,j} \Big(\alpha_{m+1}\mathrm{S}_xQ_{m+j} +(\alpha_{m+1}+\gamma_m)\frac{\texttt{U}_1}{\alpha}\mathrm{D}_xQ_{m+j} \Big).
\end{align*}
In addition, since $a'_{n,n+N}b'_{n,n+M}=b_{N,M+N+N}a_{M,M+N+n} \neq 0$ and
\begin{align*}
\alpha_{m+1}\alpha_{m+j} +\frac{\alpha^2 -1}{\alpha}(\alpha_{m+1}+\gamma_m)\gamma_{m+j}&=\alpha_{2m+j+1}\;,\\ \Big(\alpha\alpha_{m+1}-\frac{\alpha^2-1}{\alpha}(\alpha_{m+1}+\gamma_m) \Big)\gamma_{m+j}+\gamma_m\alpha_{m+j}&=\gamma_{2m+j}\;,
\end{align*}
we clearly have 
$\deg \psi_{N+k+n}=N+k+n$, $\deg \phi_{M+m+n+1}=M+m+n+1$ and $\deg \rho_{M+m+n}=M+m+n$. Thus the desired result follows.
\end{proof}
Let us now state the first result.
\begin{theorem}\label{main-result-general-for-m=k}
Let $({\bf u}, {\bf v})$ be a pair of regular functionals with respect to the pair of monic OPS $\big((P_n)_{n\geq 0},(Q_n)_{n\geq 0}\big)$. Assume that \eqref{general_problem} holds. Under the assumptions and conclusion of Lemma \ref{lemma1}, assume further that $m=k$,
\begin{align*}
\phi_{M+k+n+1}(z) \rho_{M+k+n+1}(z) - \phi_{M+k+n+2}(z) \rho_{M+k+n}(z)\neq 0\;,~ n=0,1,\ldots\;,
\end{align*}
and $\det (\mathcal{B}_4) \neq 0$, where $\mathcal{B}_4(z)=\Big[c_{i,j}(z)\Big]_{i,j=0} ^{3}$ is the following polynomial matrix of order four
\begin{align*}
c_{i,j}(z) = \left\{
\begin{array}{lclcl}
\displaystyle \texttt{U}_1\mathrm{D}_x\psi_{N+k+i}(z)-\alpha \mathrm{S}_x\psi_{N+k+i}(z), ~ \textit{if}~~j=0 \;, \\ [7pt]
\alpha\mathrm{S}_x\phi_{M+k+i+1}(z)-\texttt{U}_1(\mathrm{D}_x\phi_{M+k+i+1}(z)-K_{M+k+i}(z)), ~ \textit{if}~~j=1 \;,\\
\mathrm{D}_x\phi_{M+k+i+1}(z) +(2\alpha^2-1)K_{M+k+i}(z),~ \textit{if}~j=2\;, \\
\mathrm{D}_x \rho_{M+k+i}(z),~\textit{otherwise\;.}
\end{array}
\right. 
\end{align*}
where $K_{M+k+i}(z)=\mathrm{S}_x\rho_{M+k+i}(z)-\alpha^{-1}\texttt{U}_1\mathrm{D}_x\rho_{M+k+i} (z)$, for $i=0,1,2,3$.\\ \\
Then ${\bf u}$ and ${\bf v}$ are semiclassical functionals. That is, there exist four nonzero polynomials $\phi_1$, $\phi_2$, $\psi_1$ and $\psi_2$, such that
$$\phi_1{\bf D}_x{\bf u}=\psi_1{\bf S}_x{\bf u}\;,\quad \phi_2{\bf D}_x{\bf v}=\psi_2{\bf S}_x{\bf v} \;.$$  In addition, there exist two nonzero polynomials $\pi$ and $\rho$ such that $$\pi{\bf u}= \rho{\bf S}_x {\bf v}.$$
\end{theorem}

\begin{proof}
Since $k=m$, taking $n=0$ and $n=1$ in \eqref{relation_with_dual_basis}, we obtain
\begin{align*}
\psi_{N+k}{\bf u}&=\phi_{M+k+1}{\bf D}_x{\bf v} +\rho_{M+k}{\bf S}_x{\bf v}  \;,\\
\psi_{N+k+1}{\bf u}&=\phi_{M+k+2}{\bf D}_x{\bf v} +\rho_{M+k+1}{\bf S}_x{\bf v} \;.\\
\end{align*}
The determinant of the above system does not vanish identically by assumption, and so we have $\phi_2{\bf D}_x{\bf v}=\psi_2{\bf S}_x{\bf v}$ and $\pi{\bf u}= \rho{\bf S}_x {\bf v}$, where
$$\phi_2=\phi_{M+k+1}\psi_{N+k+1}-\phi_{M+k+2}\psi_{N+k}, ~\psi_2 =\rho_{M+k+1}\psi_{N+k} -\rho_{M+k} \psi_{N+k+1}\;,$$ $$\pi=\phi_{M+k+2}\psi_{N+k} -\phi_{M+k+1} \psi_{N+k+1},~ \rho=\phi_{M+k+2}\rho_{M+k}-\phi_{M+k+1}\rho_{M+k+1}.$$
Now we apply ${\bf D}_x$ to \eqref{relation_with_dual_basis} using \eqref{def_D_xfu} and \eqref{DxnSx-u} to obtain
\begin{align*}
\mathrm{D}_x\psi_{N+k+n}{\bf S}_x{\bf u}&=\Big(\texttt{U}_1\mathrm{D}_x\psi_{N+k+n}-\alpha \mathrm{S}_x\psi_{N+k+n} \Big) {\bf D}_x{\bf u} \\
&+\Big(\alpha\mathrm{S}_x\phi_{M+k+n+1}-\texttt{U}_1(\mathrm{D}_x\phi_{M+k+n+1}-K_{M+k+n}) \Big) {\bf D}_x ^2{\bf v} \\
&+\Big(\mathrm{D}_x\phi_{M+k+n+1} +(2\alpha^2-1)K_{M+k+n}  \Big) {\bf S}_x{\bf D}_x{\bf v} \\
&+\mathrm{D}_x\rho_{M+k+n} {\bf S}_x ^2 {\bf v}\;,
\end{align*}
for $n=0,1,\ldots$. Taking $n=0,1,2,3$, we obtain the following system 
\begin{align*}
\begin{bmatrix}
\mathrm{D}_x\psi_{N+k}{\bf S}_x{\bf u}\\
\mathrm{D}_x\psi_{N+k+1}{\bf S}_x{\bf u}\\
\mathrm{D}_x\psi_{N+k+2}{\bf S}_x{\bf u}\\
\mathrm{D}_x\psi_{N+k+3}{\bf S}_x{\bf u}
\end{bmatrix} =\mathcal{B}_4 \begin{bmatrix}
{\bf D}_x{\bf u}\\
{\bf D}_x ^2{\bf v}\\
{\bf S}_x{\bf D}_x{\bf v}\\
{\bf S}_x ^2{\bf v}
\end{bmatrix}\;.
\end{align*}
Since $B(z)=\det (\mathcal{B}_4(z))\neq 0$, then we can solve this system for ${\bf D}_x{\bf u}$, that is, there exists a nonzero polynomial $\psi_1$  such that $B {\bf D}_x{\bf u}=\psi_1{\bf S}_x{\bf u}$. The proof is done.
\end{proof}

\begin{remark}
It is important to notice that if $k>m$ in \eqref{general_problem}, then we can apply ${\bf D}_x$ to \eqref{relation_with_dual_basis} and proceed as in the previous proof to show that, under some assumptions, ${\bf u}$ is semiclassical but we can not insure that ${\bf v}$ is also semiclassical. We emphasize that the structure relation considered here has no link with the notion of coherence pair of measures and so this explain why our results are different. Regarding this we remind that the idea behind our considered relation was to characterize such OPS and therefore generalize some known results as mentioned in the introduction. Finding possible examples and connection with the so-called Sobolev OPS are not in the scope of this note and may lead to a possible future direction. In what follows we analyse closely some particular cases.
\end{remark}

For our purpose the following result is helpful.

\begin{theorem}\cite[Theorem 3.6]{KJ2018}\label{KJ_to_use_2018}
The Sturm-Liouville type equation
\begin{align}\label{Second_order_structure_relation}
\phi(z)\mathrm{D}_x ^2\, Y(z)+\psi(z) \mathrm{S}_x \mathrm{D}_x \, Y(z)=\lambda_n\, Y(z)\;,
\end{align}
where $\phi$ and $\psi$ are polynomials of degree at most two and one, respectively, and $\lambda_n$ is a constant, has a polynomial solution $P_n(z)$, of degree $n=0,1,\ldots$, for the lattice $z=4\beta s^2 +\mathfrak{c}_5s$, $\beta \neq 0$, if and only if up to a multiplicative constant, $P_n$ is the continuous dual Hahn polynomial or the Wilson polynomial. 
\end{theorem}

The following result is the analogue of \cite[Theorem 3.1]{DMAZ2022} for quadratic lattices.
\begin{theorem}\label{propo-sol-quadratic}
Consider the lattice $z=x(s):=4\beta s^2+\mathfrak{c}_5s$, with $(\beta,\mathfrak{c}_5)\neq (0,0)$. Let $(P_n)_{n\geq 0}$ be a monic OPS with respect to ${\bf u}\in \mathcal{P}^*$. Assume that the following equation holds
\begin{align}\label{equation-exple-deg-is-two}
(az^2+bz+c)\mathrm{D}_xP_n =a_n\mathrm{S}_x P_{n+1}+b_n\mathrm{S}_x P_n +c_n\mathrm{S}_x P_{n-1}\;,\quad n=0,1,\ldots\;,
\end{align}
with $c_n\neq 0$ for $n=0,1,\ldots$, where the constant parameters $a$, $b$ and $c$ are chosen such that 
\begin{align}
6aC_2C_3+2(1+a^{-1})r_3\big(C_1-\frac{\mathfrak{c}_5^2}{4} \big)+r_3\Big((B_1-B_0)^2 -8\beta (B_0+B_1-2\beta )-2C_2 \Big)=0\;,\label{condition-case-deg-is-two-1-quadratic}
\end{align}
whenever $a\neq 0$, and  
\begin{align}
aC_2C_3\Big(b_2+2aB_2+b' \Big)-r_3\Big(a(B_2+B_1)C_2 +b'C_2 -\frac{r_2}{2}(B_1-B_0) \Big)=0\;, \label{condition-case-deg-is-two-2-quadratic}
\end{align}
with $r_i=c_i+2aC_i$, $i=2,3$ and $b'=b+2a\beta$.
Then $(P_n)_{n\geq 0}$ are the multiple of continuous dual Hahn polynomials or Wilson polynomials or special or limiting cases of them. 
Moreover \eqref{Second_order_structure_relation} holds with
\begin{align}
&\phi(z)=(B_0-z)(\mathfrak{a}z+\mathfrak{b}-\mathfrak{a}B_1) -(\mathfrak{a}+1)C_1,~ \psi(z)=z-B_0,~\lambda_n=n(1+(n-1)\mathfrak{a})\;,\label{expression-phi-psi-general-case-quadratic}
\end{align}
where
\begin{align*}
&\mathfrak{a}:=\frac{aC_3}{r_3C_1}\;,~~\mathfrak{b}:=\frac{1}{2}(B_1-B_0+4\beta) \;;
\end{align*}
being $B_0$, $B_1$, $C_1$, $C_2$ and $C_3$ are coefficients for the TTRR relation \eqref{TTRR_relation} satisfied by $(P_n)_{n\geq 0}$.
\end{theorem}  

\begin{proof}
Let $(P_n)_{n\geq 0}$ be a monic OPS with respect to the functional ${\bf u} \in \mathcal{P}^*$ and satisfying \eqref{equation-exple-deg-is-two}. For the quadratic lattice $z=x(s)=4\beta s^2+\mathfrak{c}_5s$, if we set $\pi_2(z):=az^2+bz+c$, then by using \eqref{def-fDxg}, we obtain
\begin{align}\label{start-point-01}
\pi_2\mathrm{D}_x P_n= \mathrm{D}_x \left[ \Big(\mathrm{S}_x\pi_2 -2\beta \mathrm{D}_x\pi_2 \Big)P_n\right] -\mathrm{S}_x\Big( \mathrm{D}_x\pi_2~P_n  \Big)\;.
\end{align}
By direct computations we have
\begin{align*}
&\mathrm{D}_x \pi_2(z)=2 a(z+\beta)+b,~~\mathrm{S}_x\pi_2(z)=az^2+(b+6a\beta)z+\pi_2(\beta)+a\mathfrak{c}_5^2/4\;.
\end{align*} 
Hence from \eqref{start-point-01}, using \eqref{TTRR_relation} together with \eqref{equation-exple-deg-is-two}, we obtain
\begin{align*}
\sum_{j=n-1} ^{n+1} a_{n,j}\mathrm{S}_xP_j(z)=\sum_{j=n-3} ^{n+1} b_{n,j}P_{j} ^{[1]}(z) \;,\quad \quad n=0,1,\ldots\;,
\end{align*}
where
\begin{align*}
&a_{n,n+1}=2a+a_n,~a_{n,n}=b_n+2aB_n +b+2a\beta ,~a_{n,n-1}=c_n+2aC_n,\\
&b_{n,n+1}=a(n+2), ~b_{n,n}=(n+1)\Big(a(B_{n+1}+B_n)+b+2a\beta  \Big),\\
&b_{n,n-2}=(n-1)C_n\Big(a(B_{n-1}+B_n)+b+2a\beta  \Big),~b_{n,n-3}=a(n-2)C_nC_{n-1}\;,\\
&b_{n,n-1}=n\left[a\Big(C_{n+1}+B_n^2+C_n \Big)+(b+2a\beta)(B_n-\beta)\right. +c-a\Big(\beta^2-\frac{\mathfrak{c}_5 ^2}{4}  \Big)      \Big]\;.
\end{align*}
Assume $a\neq 0$ and define $$Q_n(z):=\sum_{j=n-2} ^n a_{n-1,j}P_j(z)\;,~~\quad n=0,1,\ldots\;.$$ Then $(Q_n)_{n\geq 0}$ is a simple set of polynomials and so let $({\bf a}_n)_{n\geq 0}$, $({\bf a}_n ^{[1]})_{n\geq 0}$ and $({\bf r}_n)_{n\geq 0}$ be the associated basis to the sequences $(P_n)_{n\geq 0}$, $(P_n ^{[1]})_{n\geq 0}$ and $(Q_n)_{n\geq 0}$, respectively. We then obtain
\begin{align}
&{\bf a}_n=a_{n-1,n}{\bf r}_n+a_{n,n}{\bf r}_{n+1}+a_{n+1,n}{\bf r}_{n+2}\;;\label{basis-relation-an-with-rn-1}\\
&{\bf \mathrm{S}}_x {\bf a}_n ^{[1]}=b_{n-1,n}{\bf r}_n+b_{n,n}{\bf r}_{n+1}+b_{n+1,n}{\bf r}_{n+2}+b_{n+2,n}{\bf r}_{n+3}+b_{n+3,n}{\bf r}_{n+4}\;,\label{basis-relation-with-rn-2}
\end{align}

Let us write $P_n(z)=z^n +f_nz^{n-1}+\ldots$, for $n=0,1,\ldots$, with $f_n=-\sum_{j=0} ^{n-1}B_j$.  Then using \eqref{Dx-xnSx-xn-quadratic}, we identify the coefficients of the terms in $z^{n+1}$ and in $z^n$ in each member of \eqref{equation-exple-deg-is-two} to obtain
\begin{align*}
a_n=na\;,~b_n=bn +anB_n -\frac{2}{3}an\beta (2n^2+6n+1)+a\sum_{j=0} ^{n-1}B_j\;.
\end{align*}
Without loss of generality, we set $a_{-1}:=-a$. Also taking $n=1$ in \eqref{equation-exple-deg-is-two}, we obtain
$$c_1=c-a(B_0-\beta)(B_1-\beta) +C_1 -\frac{\mathfrak{c}_5^2}{4}+(B_0-\beta)\big(b-a(6\beta-B_0-B_1)  \big)\;. $$
Then defining
\begin{align*}
A:=-\frac{a}{2}\Big((B_1-B_0)^2 -8\beta (B_0+B_1-2\beta )-2C_2\Big)-(a+1)\big(C_1 -\frac{\mathfrak{c}_5 ^2}{4} \big)\;,
\end{align*}
we combine the obtained system by taking $n=0,1,2$ in \eqref{basis-relation-an-with-rn-1} and $n=0$ in \eqref{basis-relation-with-rn-2}, using what is preceding to obtain
\begin{align}\label{basis-final}
{\bf S}_x {\bf a}_0 ^{[1]}={\bf a}_0 +\frac{1}{2}(B_1-B_0){\bf a}_1 +\frac{A}{3a}{\bf a}_2\;,
\end{align}
subject to conditions \eqref{condition-case-deg-is-two-1-quadratic}--\eqref{condition-case-deg-is-two-2-quadratic}.
Now apply ${\bf D}_x$ to \eqref{basis-final} using \eqref{DxnSx-u}, \eqref{basis-Dx-derivatives} and \eqref{def-fD_x-u} to write
$$
{\bf D}_x \left[{\bf a}_0 +\frac{1}{2}\Big((B_1-B_0)+2\texttt{U}_1\Big){\bf a}_1 +\frac{A}{3a}{\bf a}_2  \right]= - {\bf S}_x {\bf a}_1\;.
$$

So using \eqref{expression-an} and \eqref{TTRR_coefficients} we obtain
$${\bf D}_x(\phi {\bf u})={\bf S}_x(\psi {\bf u})\;,$$
where $\phi$ and $\psi$ are given in \eqref{expression-phi-psi-general-case-quadratic}. This is equivalent to the following equation
$$\phi(z) \mathrm{D}_x ^2 P_n(z) + \psi(z) \mathrm{S}_x\mathrm{D}_xP_n(z) =\lambda_nP_n(z)\;,~~\quad n=1,2,\ldots\;,$$
with $\lambda_n=n(1+(n-1)\mathfrak{a})\neq 0$, by applying \cite[Theorem 5: (a)$\Longleftrightarrow$ (c) ]{FK-NM2011}. Thus the desired result follows by Theorem \ref{KJ_to_use_2018}.
\end{proof}

\section{A special case}\label{Special_case}
Here we state a finer result for the special case where $a=0$ in \eqref{equation-exple-deg-is-two}. For this purpose the following result is appropriate.

\begin{theorem}\label{main-Thm1}\cite{KCDMJP2021-a, KDP2021}
Consider the lattice $z=x(s)=4\beta s^2+\mathfrak{c}_5s$. Let $(P_n)_{n\geq 0}$ be a monic OPS with respect to ${\bf u} \in \mathcal{P}^*$. 
Suppose that ${\bf u}$ satisfies $\mathbf{D}_x(\phi{\bf u})=\mathbf{S}_x(\psi{\bf u})$, where $\phi(z)=az^2+bz+c$ and $\psi(z)=dz+e$, with $d\neq0$.
Then $(P_n)_{n\geq 0}$ satisfies \eqref{TTRR_relation} with
\begin{align}\label{Bn-Cn-quadratic}
B_n = \frac{ne_{n-1}}{d_{2n-2}} -\frac{(n+1)e_n}{d_{2n}} -2\beta n(n-1),\;  
C_{n+1} =-\frac{(n+1)d_{n-1}}{d_{2n-1}d_{2n+1}}\phi ^{[n]}\left(-\beta n^2 -\frac{e_n}{d_{2n}}  \right),
\end{align} 
where $d_n =an+d$, $e_n=bn+e+2\beta dn^2$, and 
$$\phi ^{[n]}(z)=az^2 +(b+6\beta nd_n)z+ \phi(\beta n^2)+2\beta n\psi(\beta n^2)+\frac{n}{4}\mathfrak{c}_5 ^2d_n\;.
$$ 

\end{theorem}
Recall that the continuous monic dual Hahn polynomial $(H_n(\cdot;a,b,c))_{n\geq 0}$ satisfies \eqref{TTRR_relation} (see \cite[p.197, (9.3.5)]{KLS2010}) with
\begin{align*}
&B_n=(n+a+b)(n+a+c)+n(n+b+c-1)-a^2\;,  \\
&C_{n+1}=(n+1)(n+a+b)(n+a+c)(n+b+c)\;,  \quad\quad n=0,1,\ldots\;,
\end{align*}
with the restrictions $-a-b,-a-c,-b-c \notin \mathbb{N}_0$.

We are now in the position to state our result.
\begin{corollary} \label{main-thm-quadratic-2}
Consider the lattice  $x(s)=4\beta s^2 +\mathfrak{c}_5 s $ with $(\beta,\mathfrak{c}_5)\neq (0,0)$.
Let $(P_n)_{n\geq 0}$ be a monic OPS with respect to ${\bf u}\in \mathcal{P}^*$ satisfying
\begin{align}
(z+c)\mathrm{D}_xP_{n}(z)= b_n\mathrm{S}_xP_n(z)+c_n\mathrm{S}_x P_{n-1}(z)\;,\quad n=0,1,\ldots\;,\label{equation-case-deg-is-one-quadratic}
\end{align}
with $c_n\neq 0$ for $n=1,2,\ldots$, where the constant $c$ is chosen such that 
\begin{align}
2\left(C_2+ (B_0-B_1)c  \right)=(B_1-5\beta)^2-(B_0-5\beta)^2\;.\label{condition-case-deg-is-one-quadratic}
\end{align}
 Then up to an affine transformation of the variable, $P_n$ is one of the following specific case of the continuous dual Hahn polynomial $(H_n(\cdot;a,b,c))_{n\geq 0}$:
\begin{align}\label{solu-case-deg-one-quadratic}
P_n(z)=(-4\beta)^nH_n\left(-\frac{1}{4\beta}\big( z +\frac{\mathfrak{c}_5 ^2}{16\beta}\big);a,b,\frac{1}{2}  \right)\;,
\end{align}
or
\begin{small}
\begin{align}\label{solu-case-deg-one-quadratic-additional}
P_n(z)=(-4\beta)^nH_n\left(-\frac{1}{4\beta}\big( z+2\beta +\frac{\mathfrak{c}_5 ^2}{16\beta}\big);d-\frac{1}{2},e-\frac{1}{2},\frac{1}{2}  \right),
\end{align}
\end{small}
 for each $n=0,1,\ldots$, where $d$ and $e$ are complex numbers such  $-d\notin \mathbb{N}_0$ and 
$$
e:=-1+\frac{1}{1+d}\left(1-\frac{\mathfrak{c}_5 ^2 }{64\beta ^2}   \right).
$$
\end{corollary}

\begin{proof}
Let ${\bf u}\in \mathcal{P}^*$ be the regular functional with respect to which $(P_n)_{n\geq 0}$ is a monic OPS. Suppose that $(P_n)_{n\geq 0}$ satisfies \eqref{equation-case-deg-is-one-quadratic} and subject to the restriction \eqref{condition-case-deg-is-one-quadratic}. Then from Theorem \ref{propo-sol-quadratic} we deduce (by taking $a=0$, $b=1$ and taking also $n=1,2$ in \eqref{equation-case-deg-is-one-quadratic}) that ${\bf D}_x(\phi {\bf u})={\bf S}_x(\psi {\bf u})$ holds where
\begin{align*}
&\phi(z)= -\frac{1}{2}\left(B_1-B_0+4\beta \right)(z-B_0)-C_1,~~\psi(z)=z-B_0\;.
\end{align*}
We apply \eqref{Bn-Cn-quadratic} to obtain
\begin{align}
B_n&=-8\beta n^2+(B_1-B_0+8\beta)n+B_0\;,\label{bn-case-degree-is-1-quadratic}     \\
C_{n+1}&=\frac{1}{4}(n+1)\Big[64\beta ^2 n^3 -16\beta (B_1-B_0+4\beta)n^2 +4C_1 \label{cn-case-degree-is-1-quadratic}\\
&\quad\quad \quad\quad\quad\quad~~~~~~~~~ +\Big((B_1-B_0)^2-\mathfrak{c}_5 ^2 +8\beta(B_1-3B_0+2\beta )  \Big)n  \Big]  \nonumber \;,
\end{align}
for each $n=0,1,\ldots$. In addition, we claim that the parameters $B_0$, $B_1$ and $C_1$ are related by the following equation
\begin{align}
2\beta C_1=(B_1-	B_0+8\beta)\big((B_0-\beta)\beta +\frac{\mathfrak{c}_5 ^2}{16}\big)\;,\label{C1-in-term-of-B0-B1}
\end{align}
with $B_0+B_1=2\beta$ or $B_0+B_1\neq 2\beta$.

Indeed writing $P_n(z)=z^n +f_nz^{n-1}+\cdots$, where $f_0:=0$ and $f_n=-\sum_{j=0} ^{n-1}B_j$ for $n=1,2,\ldots$, we identify the coefficients of the two firsts therms with higher degree in \eqref{equation-case-deg-is-one-quadratic} using \eqref{Dx-xnSx-xn-quadratic} to obtain 
\begin{align}
b_n=n,\quad c_n=\sum_{j=0} ^{n-1}B_j -\frac{1}{3}(4n^2-1)n\beta +nc\;, \quad n=0,1,\ldots \;.\label{second-hand-datta-case-deg-1-quadratic}
\end{align}

Also by direct computations we obtain
\begin{align*}
&D_xP_2(z)=2z+2\beta -B_0-B_1\;,\\
&S_xP_2(z)=z^2+(6\beta-B_0-B_1)z+(\beta-B_0)(\beta-B_1) +\frac{\mathfrak{c}_5 ^2}{4}-C_1\;.
\end{align*}
Similarly we obtain $D_xP_3$ and $S_xP_3$ by taking $n=3$ in \eqref{TTRR_relation} using \eqref{def-Dx-fg}--\eqref{def-Sx-fg}:
\begin{align*}
D_xP_3(z)&=3z^2+2(5\beta-B_0-B_1-B_2)z+(\beta-B_2)(2\beta-B_0-B_1)\\
&+(\beta-B_0)(\beta-B_1)-C_1-C_2  +\frac{\mathfrak{c}_5 ^2}{4} \;,
\end{align*}
and
\begin{small}
\begin{align*}
S_xP_3(z)&=z^3+(15\beta-B_0-B_1-B_2)z^2+\Big(4\beta(2\beta-B_0-B_1)+3\frac{\mathfrak{c}_5 ^2}{4}-C_1-C_2\\
&+(\beta-B_0)(\beta-B_1)+(\beta-B_2)(6\beta-B_0-B_1) \Big)z+(B_0-\beta)C_2\\
&-(B_0+B_1+B_2-3\beta)\frac{\mathfrak{c}_5 ^2}{4}+(\beta-B_2)\Big((\beta-B_0)(\beta-B_1)-C_1 \Big)   \;.
\end{align*}
\end{small}
Taking $n=3$ in \eqref{equation-case-deg-is-one-quadratic}, using what is preceding we obtain the following equations:
\begin{small}
\begin{align*}
C_2+C_1-2\beta(c_3-c)+(B_0+B_1-2\beta)(\frac{1}{2}c_3+7\beta-B_2-c)\\
+(B_2-\beta)(6\beta-c)+(B_0-\beta)(\beta-B_1)-\mathfrak{c}_5 ^2=0\;, \\
(3B_0-3\beta+c)C_2+(3B_2-3\beta-c_3+c)C_1+(\beta-B_0)(\beta-B_1)(c_3+3\beta-3B_2-c)\\
-\frac{\mathfrak{c}_5 ^2}{4}(3B_0+3B_1+3B_2-9\beta-c_3+c)+c(B_2-\beta)(2\beta-B_0-B_1)=0\;.
\end{align*}
\end{small}
Hence \eqref{C1-in-term-of-B0-B1} is obtained from the previous equations by using the expressions of $c_3$, $B_2$, $C_2$ and $c$ giving by \eqref{second-hand-datta-case-deg-1-quadratic}, \eqref{bn-case-degree-is-1-quadratic}, \eqref{cn-case-degree-is-1-quadratic} and \eqref{condition-case-deg-is-one-quadratic}, respectively.   
\\
Recall that from \eqref{condition-case-deg-is-one-quadratic}, $B_1\neq B_0$ and consequently from \eqref{C1-in-term-of-B0-B1}, we obtain $\beta \neq 0$. Further, according to \eqref{C1-in-term-of-B0-B1}, only $B_0$ and $B_1$ may be consider as free parameters. Let then $a$ and $b$ be two complex numbers solutions of the following quadratic equation:

$$8\beta Z^2 +(B_1-B_0+8\beta)Z +\frac{1}{2} (B_1-5B_0)+4\beta -\frac{\mathfrak{c}_5 ^2}{8\beta}  =0\;.$$
That is 
$$
(a,b) ~\textit{or}~(b,a)\in \left\lbrace \Big(\frac{1}{16\beta}(B_0-B_1-8\beta)-\sqrt{\Delta} ,\frac{1}{16\beta}(B_0-B_1-8\beta)+\sqrt{\Delta} \Big)   \right\rbrace
\;,$$
where $\Delta:=\frac{1}{256\beta^2}(B_1-B_0+8\beta)^2 +\frac{1}{8\beta}\Big(\frac{1}{2}(5B_0-B_1)-4\beta +\frac{\mathfrak{c}_5 ^2}{8\beta}  \Big)$.
Then we may express $B_0$ and $B_1$ in term of $a$ and $b$ as follows
\begin{align*}
&B_0=-2\beta (a+b+2ab) -\frac{\mathfrak{c}_5 ^2}{16\beta}  \;,\\
&B_1=-2\beta (5a+5b+4+2ab) -\frac{\mathfrak{c}_5 ^2}{16\beta}  \;.
\end{align*}
So \eqref{C1-in-term-of-B0-B1} becomes $C_1=4\beta^2(a+b)(2a+1)(2b+1)$.
Then from \eqref{bn-case-degree-is-1-quadratic}--\eqref{cn-case-degree-is-1-quadratic} we obtain
\begin{align*}
&\small{B_n=-2\beta \Big(4n(n+a+b)+2ab+a+b  \Big) -\frac{\mathfrak{c}_5 ^2}{16\beta}}\;,\\
&C_{n+1}=4\beta^2(n+a+b)(2n+2a+1)(2n+2b+1)\;,\quad \quad n=0,1,\ldots\;,
\end{align*}
with the condition $-a-b,-(2a+1)/2,-(2b+1)/2 \notin \mathbb{N}_0$, obtained from the regularity conditions. In addition, using \eqref{condition-case-deg-is-one-quadratic} and \eqref{second-hand-datta-case-deg-1-quadratic}, we obtain
\begin{align*}
c=\frac{\mathfrak{c}_5^2}{16\beta}\;,\quad c_n=&-\beta n(2n+2a-1)(2n+2b-1)\;.
\end{align*}
Hence \eqref{solu-case-deg-one-quadratic} holds.

We remark that if in addition to \eqref{C1-in-term-of-B0-B1}, we have $B_1=-B_0+2\beta$, then $B_0$ will be the only free parameter. For this case let $d$ and $e$ be two complex numbers solutions of the following quadratic equation
$$
4\beta Z^2+(\beta-B_0)(Z+1)-\frac{\mathfrak{c}_5 ^2}{16\beta}=0\;.
$$
Then we obtain 
$$
(d,e)~\textit{or}~(e,d)\in \left\lbrace\left(\frac{1}{8\beta}(B_0-\beta)-\sqrt{\Delta},\frac{1}{8\beta}(B_0-\beta)+\sqrt{\Delta}   \right)    \right\rbrace\;,
$$
where $\Delta:=\frac{1}{4\beta}\left[(B_0-\beta)\left(1+\frac{1}{16\beta}(B_0-\beta)\right)+\frac{\mathfrak{c}_5 ^2}{16\beta} \right]$.

Then we have $\beta (1+4d+4e)=B_0=\beta -4\beta de-\frac{\mathfrak{c}_5 ^2}{16\beta}$. So $d$ and $e$ are related by the following relation
$$
(d+1)(e+1)=1-\frac{\mathfrak{c}_5 ^2}{64\beta ^2}  \;.
$$
Therefore \eqref{bn-case-degree-is-1-quadratic}--\eqref{cn-case-degree-is-1-quadratic} become
\begin{align*}
B_n&=-8\beta n^2 -8\beta(d+e-1)n +\beta(1-4de)-\frac{\mathfrak{c}_5 ^2}{16\beta} \;, \\
C_{n+1}&=16\beta^2(n+1)(n+d+e-1)(n+d)(n+e)\;,\quad n=0,1,\ldots\;,
\end{align*}
where $-d,-e,-d-e\notin \mathbb{N}_0$ by regularity conditions. We also have 
$$
c=\frac{\mathfrak{c}_5 ^2}{16\beta},~c_n=-\frac{4}{3}n\beta\left(7n^2+6(d+e-1)n+3de-1  \right)
$$
Hence \eqref{solu-case-deg-one-quadratic-additional} holds.
\end{proof}

\section*{Acknowledgements }
The authors are partially supported by ERDF and Consejer\'ia de Econom\'ia, Conocimiento, Empresas y Universidad de la Junta de Andaluc\'ia (grant UAL18-FQM-B025-A) and by the Research Group FQM-0229 (belonging to Campus of International Excellence CEIMAR). The authors JFMM and JJMB are partially supported by the Ministry of Science, Innovation, and Universities of Spain and the European Regional Development Fund (ERDF) (Grant MTM2017-89941-P) and by the Research Centre CDTIME of Universidad de Almer\'ia. The author JJMB is also supported in part by Junta de Andaluc\'ia and ERDF, Ref. SOMM17/6105/UGR.

{

\end{document}